\documentclass[12pt, twoside, leqno]{article}

\usepackage{amsmath,amsthm}
\usepackage{amssymb}
\usepackage{hyperref}

\usepackage{enumerate}

\usepackage{graphicx}

\newtheorem{thm}{Theorem}[section]
\newtheorem{cor}[thm]{Corollary}
\newtheorem{lem}[thm]{Lemma}
\newtheorem{prop}[thm]{Proposition}
\newtheorem{prob}[thm]{Problem}

\theoremstyle{definition}
\newtheorem{defin}[thm]{Definition}
\newtheorem{rem}[thm]{Remark}

\numberwithin{equation}{section}

\newcommand{\eps}{\varepsilon}

\begin{document}


\baselineskip=17pt


\title{On Qian's problem for $\mathcal{L}_{\infty}$-spaces}

\author{Duanxu Dai\\
College of Mathematics and Computer Science\\
Quanzhou Normal University\\
Quanzhou 362000, China\\
E-mail: dduanxu@163.com
\and}

\date{}

\maketitle


\renewcommand{\thefootnote}{}



\footnote{ Supported by the Natural Science Foundation of China (Grant No. 11601264) and the Outstanding Youth Scientific Research Personnel Training Program of Fujian Province and the High level Talents Innovation and Entrepreneurship Project of Quanzhou City and the Research Foundation of Quanzhou Normal University(Grant No. 2016YYKJ12). }
\footnote{}

\renewcommand{\thefootnote}{\arabic{footnote}}
\setcounter{footnote}{0}


\begin{abstract}
  In this paper we devote to study Qian's problem for $\mathcal{L}_{\infty}$-spaces. Firstly, a positive answer to Qian's problem for $C(K)$-spaces is given by the assumption that $K$ has the C$\check{e}$ch-Stone property. Secondly, we obtain quantitative characterizations of separably injective spaces that turn out to give a positive answer to Qian's problem of 1995 in the setting of separable universality. Thirdly, we prove a sharpen quantitative and generalized Sobczyk theorem, which gives sharpen constants ($\alpha,\gamma$) for Qian's Problem. Finally, we give a more generalized Figiel theorem for $\mathcal{L}_{\infty}$-spaces.
\end{abstract}

\section{Introduction}

Mazur and Ulam \cite{ma} in 1932 proved that every surjective isometry between two Banach spaces $X$ and $Y$ is necessarily affine. Since then, properties of isometries and generalizations there of between Banach spaces has continued for 86 years. On this period, many significant problems about perturbation properties of surjective $\eps$-isometries were proposed and solved by numerous mathematicians. In particular, we mention the Hyers-Ulam problem \cite{hy} (see, for instance, \cite{ge}, \cite{gru}, and  \cite{om}). In 1968, Figiel \cite{figiel} showed the following remarkable result(Figiel theorem): For every standard isometry $f:X\rightarrow Y$  there is a linear operator $T:L(f)\rightarrow X$ with $\|T\|=1$ so that $Tf=Id$ on $X$, where $L(f)$ is the closure of span $f(X)$ in $Y$ (see also \cite{ben} and \cite{du}).

\begin{defin} Let $X,Y$ be two Banach spaces, $\eps\geq0$, and let $f:X\rightarrow Y$ be a mapping.

(1) $f$ is said to be an $\eps$-isometry if
\begin{align} |\|f(x)-f(y)\|-\|x-y\||\leq\eps\;\;{\rm for\; all}\; x,y\in X.
\end{align}
In particular, a $0$-isometry $f$ is simply called an isometry.

(2) We say  an $\eps$-isometry $f$ is standard if $f(0)=0$.

(3) A standard $\eps$-isometry $f$ is $(\alpha,\gamma)$-stable if there exist $\alpha, \gamma>0$ and a bounded linear operator $T:L(f)\rightarrow X$ with $\|T\|\leq\alpha$ such that
\begin{align}\label{E:1.2}\|Tf(x)-x\|\leq\gamma\eps,\;\;{\rm for\;all\;}x\in X.
\end{align}
In this case, we also simply say $f$ is stable, if no confusion arises.

(4) A pair $(X,Y)$ of Banach spaces $X$ and $Y$ is said to be stable  if
   every standard $\eps$-isometry $f:X\rightarrow Y$ is $(\alpha,\gamma)$-stable for some $\alpha,\gamma>0$.

(5) A pair $(X,Y)$ of Banach spaces $X$ and $Y$  is called  $(\alpha,\gamma)$-stable for some $\alpha,\gamma>0$ if every standard $\eps$-isometry $f:X\rightarrow Y$ is $(\alpha,\gamma)$-stable.
\end{defin}

The study of non-surjective $\eps$-isometries has also been considered (see, for instance, \cite{bao}, \cite{cheng}, \cite{cheng2}, \cite{Dai}, \cite{dil}, \cite{om}, \cite{qian}, \cite{sm} and \cite{ta}).
Qian\cite{qian} proposed the following problem in 1995.

\begin{prob} \label{P1} Is it true that for every pair $(X,Y)$ of Banach spaces $X$ and $Y$ there exists $\gamma>0$ such that every standard
$\varepsilon$-isometry $f:X\rightarrow Y$ is $(\alpha,\gamma)$-stable for some $\alpha>0$?
\end{prob}
\noindent

However, Qian  \cite{qian} presented a counterexample showing that if a separable Banach space $Y$ contains an uncomplemented closed subspace $X$ then for every $\eps>0$ there is a standard $\eps$-isometry $f:X\rightarrow Y$ which is not stable. Recently, Cheng et al \cite{cheng3} showed the following sharp weak stability version.

\begin{thm}[Cheng et al]\label{T:1.3}
 Let $X$ and $Y$ be Banach spaces,
  and let $f:X\rightarrow Y$ be a standard $\eps$-isometry for
 some $\eps\geq 0$. Then for every $x^*\in X^*$, there exists $\phi\in Y^*$ with $\|\phi\|=\|x^*\|\equiv r$ such that
 \begin{align*} |\langle\phi,f(x)\rangle-\langle x^*,x\rangle|\leq2\eps r, \; for\;all\;x\in X.\end{align*}
 \end{thm}

For study of the stability of $\eps$-isometries of Banach spaces, the following question was proposed in \cite{cheng2}.
\begin{prob}
Is there a  characterization for the class of Banach spaces $\mathcal{X}$ satisfying given any $X\in \mathcal{X}$ and Banach space $Y$, the pair $(X,Y)$ is $($$(\alpha,\gamma)$-, resp.$)$ stable?
\end{prob}

Every space $X$ of this class is said to be a universally ($(\alpha,\gamma)$-, resp.) left-stable space.
On one hand, Cheng, Dai, Dong et.al. \cite{cheng2} proved that every injective Banach space is a universally left-stable space. On the other hand, the first two authors Cheng and Dai, together with others \cite{bao} showed that every universally left-stable space is just a cardinality injective Banach space (i.e., a Banach space which is complemented in every superspace with the same cardinality) and they also showed that a dual space is injective (i.e., $X$ is complemented in $\ell_\infty(B_{X^*})$)if and only if it is a universally left-stable space.

This paper devotes to study Qian's problem for $\mathcal{L}_{\infty}$-spaces as follows. In Section 3, we obtain a weak positive answer to Qian's problem for $C(K)$-spaces (see Corollary \ref{main3}). Then by assuming that $K$ has the C$\check{e}$ch-Stone property, a positive answer for such a $C(K)$-space is given.
The following Problem \ref{P} is also very natural.

\begin{prob}\label{P}
Is there a  characterization for the class of Banach spaces $\mathcal{S}$ satisfying given any $X\in \mathcal{S}$ and separable Banach space $Y$, the pair $(X,Y)$ is $((\alpha,\gamma)$-, resp.$)$ stable?
\end{prob}

Every space $X$ of this class is said to be a separably universally (resp. $(\alpha,\gamma)$) left-stable space. In Section 4, we will show that all of these spaces of the class $\mathcal{S}$ coincide with separably injective Banach spaces. We here refer the reader to a very excellent paper \cite{ASC} by Avil$\acute{e}$s-S$\acute{a}$nchez-Castillo-Gonz$\acute{a}$lez-Moreno for further information about injective Banach spaces and separably injective Banach spaces where they resolved (under an additional assumption) a long standing problem proposed by Lindenstrauss in the middle sixties.

The following theorem was proved by Sobczyk \cite{sob} which says that $c_0$ is separable separably injective space( A Banach space $X$ is said to be $\lambda$-separably injective if it has the following extension property: Every bounded linear operator $T$ from a closed subspace of a  separable Banach space into $X$ can be extended to be a bounded operator on the whole space with its norm at most $\lambda\|T\|$. In this case, $X$ is said to be separably injective if it is $\lambda$-separably injective for some $\lambda\geq 1$ \cite{zip}) while Zippin \cite{zip} showed that $c_0$ is the unique separable separably injective space, up to an isomorphism.
In 2014, Cheng, Dai et al \cite{cheng2} proved that $c_0$, up to an isomorphism, is the unique separable space such that the couple $(c_0, Y)$ is stable for every separable space $Y$.

\begin{thm}[Sobczyk theorem \cite{sob}]\label{T:1.4}
Let $X$ be a separable Banach space. If $E$ is a closed subspace of $X$ and $T: E\rightarrow c_0$ is a bounded operator then there exists an operator $\widetilde{T}: X \rightarrow c_0$ such that $\widetilde{T}|_E=T$ and $\|\widetilde{T}\|\leq 2\|T\|$.
\end{thm}

In section 5, we also prove a sharpen quantitative and generalized Sobczyk theorem (see Theorem \ref{T:1.4}), that is, Theorem \ref{T:3.7}, which gives examples of nonseparable separably injective spaces $X$ (but not injective, i.e., $X$ is not complemented in $\ell_\infty(B_{X^*})$) such that for some sharpen constants $\alpha,\gamma>0$, the couple $(X,Y)$ is $(\alpha,\gamma)$-stable for every separable space $Y$. In Section 6, we prove a more generalized Figiel theorem for $\mathcal{L}_{\infty,\lambda}$-spaces ( see \cite{AH}, \cite {ASC}, \cite{Bou}).

All symbols and notations in this paper are standard. We use $X$ to denote a real Banach space and $X^*$ its dual. $B_X$, ext $(B_{X^*})$ and $S_X$ denote the closed unit ball of $X$, the set of all extremal points of $B_{X^*}$ and the unit sphere of $X$, respectively. For a subset $A\subset X$, $\overline{A}$ and card $(A)$ stand respectively for the closure of $A$, the cardinality of $A$. Given a bounded linear operator $T:X\rightarrow Y$, $T^*:Y^*\rightarrow X^*$ stands for its conjugate operator. We denote by $d(X,Y)=\inf\{\|T\|\cdot\|T^{-1}\|: T\;\text {is\;an\;isomorphism\; between}\;X\;\text{and} \;Y \}$ the Banach-Mazur distance between $X$ and $Y$.

\section{Preliminaries }

Recall that a Banach space $X$ is said to be $\lambda$-(resp. separably injective) injective if it has the following extension property:
Every bounded linear operator $T$ from a closed subspace of a (resp. separable) Banach space into $X$ can be extended to be a bounded operator on the whole space with its norm at most $\lambda\|T\|$ (see, for instance, \cite{Alb}, \cite{ASC}, \cite{fa}, \cite{Wo}, \cite{zip}). In this case, $X$ is said to be injective (resp. separably injective) if it is $\lambda$-(resp. separably injective) injective for some $\lambda\geq 1$.

The following Proposition \ref{P:3.1} follows easily from Remark \ref{R:3.2}.

\begin{prop}\label{P:3.1} A (resp. separable) Banach space $X$ is $\lambda$-(resp. separably injective) injective if and only if it is $\lambda$-complemented in every (resp. separable) superspace (i.e., a normed linear space which contains $X$).
\end{prop}

The following Proposition \ref{P:3.2} was proved by Avil$\acute{e}$s, S$\acute{a}$nchez, Castillo, Gonz$\acute{a}$lez and Moreno (see \cite[Prop. 3.2]{ASC}).

\begin{prop}\label{P:3.2}\rm
(1) If a Banach space $X$ is $\lambda$-separably injective, then it is $3\lambda$-complemented in every superspace $Y$ such that $Y/X$ is separable.

(2) If a Banach space $X$ is $\lambda$-complemented in every superspace $Y$ such that $Y/X$ is separable, then $X$ is $\lambda$-separably injective.
\end{prop}

\begin{rem}\label{R:3.2}
For any set $\Gamma$, that  $\ell_\infty(\Gamma)$ is $1$-injective follows from the Hahn-Banach theorem.
\end{rem}

Recall that $\mathcal{S}$ is the class of Banach spaces satisfying given any $X\in \mathcal{S}$ and separable Banach space $Y$, the pair $(X,Y)$ is ($(\alpha,\gamma)$-, resp.) stable. Every space $X$ of this class is said to be a separably universally ($(\alpha,\gamma)$-, resp.) left-stable space.
In section 3, we completely solve Problem \ref{P}. That is, we prove that all of these spaces of the class $\mathcal{S}$ coincide with separably injective Banach spaces.

\begin{lem}\label{main2}
Suppose that $X$, $Y$ are Banach spaces. Let $\varepsilon\geq 0$. Assume that $f$ is a $\varepsilon-$ isometry
from $X$ into $Y$ with $f(0)=0$. Then for every $w^*$-dense subset $\Omega\subset$ ext $(B_{X^*})$ there is a bounded linear operator $T: Y\rightarrow \ell_\infty(\Omega)$ such that
$$\|T f(x)-x\|\leq 2\varepsilon,\;\;\text{for\;all}\;x\in X .$$

\end{lem}

\begin{proof} By Theorem \ref{T:1.3}, for every $x^*\in \Omega$, there exists a functional $Q(x^*)\in S_{Y^*}$ such that
\begin{align*} |\langle Q(x^*),f(x)\rangle-\langle x^*,x\rangle|\leq2\eps , \; for\;all\;x\in X.\end{align*}
We now define a mapping $T:Y\rightarrow \ell_\infty(\Omega)$ by
$$T(y)=\{Q(x^*)(y)\}_{x^*\in\Omega}.$$
It is clear that $T$ is a bounded linear operator with norm one and
$$\|T f(x)-x\|=\sup_{x^*\in\Omega}|Q(x^*)f(x)-x^*(x)|\leq 2 \varepsilon,\;\;\text{for\;all}\;x\in X .$$

\end{proof}

The following Lemma \ref{L:1.1} follows from Qian's counterexample in \cite{qian} (see also \cite{cheng2}).

\begin{lem}\label{L:1.1}
Let $X$ be a closed subspace of Banach space $Y$. If $\text{card}\;(X)=\text{card}\;(Y)$, then for every $\eps>0$ and every bijective mapping $g: X\rightarrow B_Y$ with $g(0)=0$, there is a standard $\eps$-isometry $f:X\rightarrow Y$ defined for all $x\in X$ by $f(x)=x+\frac{\varepsilon}{2}g(x)$
such that

(1)  $L(f)\equiv\overline{\rm {span}}\;f(X)=Y$;

(2)  $X$ is $\lambda$ complemented in $Y$ whenever $f$ is $(\lambda,\gamma)$ stable for some $\lambda$, $\gamma>0$.
\end{lem}

\section{On Qian's problem for $C(K)$-spaces }

Recall that a dual Banach space $Y^*$ is said to have the point of weak star to norm continuity property (in short, $w^*$-PCP) if every nonempty bounded subset of $Y^*$ admits relative weak star neighborhoods of arbitrarily small norm diameter. For example, if $Y$ is an Asplund space, then $Y^*$ has the $w^*$-PCP (see, for instance, \cite{Phe}).

Recall that a set valued mapping $F:X \rightarrow 2^{Y}$ is said to be
usco provided it is nonempty compact valued and upper
semicontinuous, i.e., $F(x)$ is nonempty compact for each $x\in X$
and $\{x\in X: F(x)\subset U\}$ is open in $X$ whenever $U$ is open
in $Y$. We say that $F$ is usco at $x\in X$ if $F$ is nonempty
compact valued and upper semicontinuous at $x$, i.e., for every open
set $V$ of $Y$ containing $F(x)$ there exists a open neighborhood $U$
of $X$ such that $F(U)\subset V$. Therefore, $F$ is usco if and only
if $F$ is usco at each $x\in X$.

Recall that a mapping $\varphi: X \rightarrow Y$ is called a selection of $F$
if $\varphi(x)\in F(x)$ for each $x\in X$, moreover, we say
$\varphi$ is a continuous (linear) selection of $F$ if $\varphi$ is
a continuous (linear) mapping. We denote the graph of $F$ by
$G(F)\equiv\{(x,y)\in X\times Y:y\in F(x)\}$, we write $F_1\subset
F_2$ if $G(F_1)\subset G(F_2)$. A usco mapping $F$ is said to be minimal
if $E=F$ whenever $E$ is a usco mapping and $E \subset F$ (see, for instance, \cite {Dai}, \cite[page 19, 102-109]{Phe}).

The following Problem \ref{P2} is equivalent to Problem \ref{P1}.

\begin{prob}\label{P2}Does there exist a constant $\gamma>0$ depending
 only on $X$ and $Y$ with the following property: For each $\varepsilon$-isometry $f:$ $X\rightarrow Y$ with
$f(0)=0$ there is a $w^*-w^*$ continuous linear selection $Q$ of the
set-valued mapping $\Phi$ from $X^*$ into $2^{L(f)^*}$ defined by
$$\Phi(x^*)=\{\phi\in L(f)^*:|\langle
\phi,f(x)\rangle-\langle x^*,x\rangle|\leq\gamma\|x^*\|\varepsilon
,\;\;\text{for\;all}\;x\in X \},$$ where $L(f)$ $=$ $\overline{\text{span}}\;f(X)$?
\end{prob}

The following Lemma \ref{main1} was motivated by Dai et.al. in \cite[Lemma 4.2]{Dai}. By an analogous argument we conclude the result on $w^*-w^*$ usco mappings, which will be used to prove Corollary \ref{main3}.

\begin{lem}\label{main1}
Suppose that $X$, $Y$ are Banach spaces. Let $\varepsilon\geq 0$. Assume that $f$ is a $\varepsilon-$ isometry
from $X$ into $Y$ with $f(0)=0$, $H$ itself is a Baire subspace contained in $S_{X^*}$ with respect to $w^*$-topology. If  we define a set-valued
mapping $\Phi_1:S_{X^*}\rightarrow 2^{S_{L(f)^*}}$ by
$$\Phi_1(x^*)=\{\phi\in S_{L(f)^*}:|\langle
\phi,f(x)\rangle-\langle x^*,x\rangle|\leq 4 \varepsilon
,\;\;\text{for\;all}\;x\in X \},$$ where $L(f)$ $=$ $\overline{\rm{span}}\,f(X)$, then

(i) $\Phi_1$ is convex $w^*$-usco at
each point of $S_{X^*}$.

(ii) There exists a minimal convex $w^*-w^*$ usco mapping contained in
$\Phi_1$.

(iii) If, in addition, $Y^*$ has the $w^*$-PCP (especially, if $Y$ is an Asplund space) or $Y$ is separable, then there exists a selection $Q$ of $\Phi_1$ such that $Q$ is $w^*-w^*$ continuous on a $w^*$-dense $G_\delta$ subset of $H$.
\end{lem}

\begin{proof} (i) It follows easily from \cite[Lemma 4.2 (i)]{Dai}.

 (ii)By Zorn Lemma (see \cite[Lemma 4.2 (ii)]{Dai} or \cite[Prop.7.3, p.103]{Phe}) there exists a minimal convex $w^*-w^*$ usco mapping contained in $\Phi_1$.

(iii) By (ii) there is a minimal convex $w^*-w^*$ usco mapping $F\subset\Phi_1$, and $H$ itself is a Baire space with respect to $w^*$-topology, and $Y^*$ has the $w^*$-PCP (especially, if $Y$ is an Asplund space) or $Y$ is separable,
which follows easily from \cite[Lemma 7.14, p.106-107]{Phe} and \cite[Lemma 4.2 (iii)]{Dai}.
\end{proof}

\begin{rem} The above Lemma \ref{main1} also holds if we substitute $Y^*$ and $S_{Y^*}$ for $L(f)^*$ and $S_{L(f)^*}$, respectively.
\end{rem}

\begin{lem} \label{main2}
Suppose that $X$, $Y$ are Banach spaces. Let $\varepsilon\geq 0$. Assume that $f$ is a $\varepsilon-$ isometry
from $X$ into $Y$ with $f(0)=0$. Then

(1) for every $w^*$-dense subset $\Omega\subset$ ext $(B_{X^*})$ there is a bounded linear operator $T: Y\rightarrow \ell_\infty(\Omega)$ such that
$$\|T f(x)-x\|\leq 2\varepsilon,\;\;\text{for\;all}\;x\in X .$$

(2) If $Y^*$ has the $w^*$-PCP or $Y$ is separable, then there exists a $w^*$-dense $G_\delta$ subset $\Omega\subset$ ext $B_{X^*}$ such that there is a bounded linear operator $T: Y\rightarrow C(\Omega)$ such that

$$\|T f(x)-x\|\leq 2 \varepsilon,\;\;\text{for\;all}\;x\in X .$$
\end{lem}

\begin{proof} (1) By Theorem \ref{T:1.3}, for every $x^*\in \Omega$, there exists a functional $Q(x^*)\in S_{Y^*}$ such that
\begin{align*} |\langle Q(x^*),f(x)\rangle-\langle x^*,x\rangle|\leq2\eps , \; for\;all\;x\in X.\end{align*}
We now define a mapping $T:Y\rightarrow \ell_\infty(\Omega)$ by
$$T(y)=\{Q(x^*)(y)\}_{x^*\in\Omega}.$$
It is clear that $T$ is a bounded linear operator with norm one and
$$\|T f(x)-x\|=\sup_{x^*\in\Omega}|Q(x^*)f(x)-x^*(x)|\leq 2 \varepsilon,\;\;\text{for\;all}\;x\in X .$$

(2) Since ext $(B_{X^*})$ itself is a Baire space in its relative $w^*$-topology (see \cite[p.217, line 17-19 ]{Hol}), it follows from Lemma \ref{main1} that there is a $w^*-$ dense $G_\delta$ subset $\Omega$ in ext $(B_{X^*})$ such that there is a $w^*-w^*$ continuous selection $Q$ of $\Phi_1$ on $\Omega$ satisfying that for every $x\in X$ and $x^*\in \Omega$, the following inequality holds :
\begin{align*} |\langle Q(x^*),f(x)\rangle-\langle x^*,x\rangle|\leq2\eps.\end{align*}
Let $T:Y\rightarrow \ell_\infty(\Omega)$ be defined as in (i).
Therefore, $ T(y)\in C(\Omega)$ and

$$\|T f(x)-x\|\leq 2\varepsilon,\;\;\text{for\;all}\;x\in X .$$
\end{proof}

\begin{cor}\label{main3}
Suppose that $X=C(K)$ for a compact Hausdorff space $K$ and $Y^*$ has the $w^*$-PCP (especially, if $Y$ is an Asplund space) or $Y$ is separable. Let $\varepsilon\geq 0$. Assume that $f$ is a standard $\varepsilon$-isometry
from $X$ into $Y$. Then there exists a dense $G_\delta$ subset $\Omega$ of $K$ such that there is a bounded linear operator $T: Y\rightarrow C(\Omega)$ such that
$T f-Id$ is uniformly bounded by $2\varepsilon$ on $X$.

\end{cor}

\begin{proof}

It suffices to note that ext $(B_{X^*})=\{\pm\delta_t:t\in K\}$ and $\{\delta_t:t\in K\}$ is a compact Baire space norming for $X$, and then apply Lemma \ref{main1} and Lemma \ref{main2} to conclude the results we desired by substituting $\{\delta_t:t\in K\}$ respectively for $H$ and ext $(B_{X^*})$ everywhere.

\end{proof}

Now we introduce a new notion the so called the C$\check{e}$ch-Stone property that a topological space $K$ is said to have the C$\check{e}$ch-Stone property provided that for every $G_\delta$ dense subset $S\subset K$ there is a dense subset $\Omega\subset S$ such that $K$ is the C$\check{e}$ch-Stone compactification of $\Omega$.
For example, $\beta\mathbb{N}$, the C$\check{e}$ch-Stone compactification of $\mathbb{N}$ has the C$\check{e}$ch-Stone property since every $G_\delta$ dense subset of it must contain the discrete space $\mathbb{N}$.

As a consequence of Corollary \ref{main3} we have

\begin{thm}Suppose that $X=C(K)$ where $K$ is a compact Hausdorff space and admitting the C$\check{e}$ch-Stone property. If either $Y^*$ has the $w^*$-PCP or $Y$ is separable, then the pair $(X,Y)$ is $(1,2)$-stable.
\end{thm}

\section{A quantitative characterization of separably injective Banach spaces}

In this section, we combine Lemma \ref{main2} with some results from \cite{John} by Johnson-Oikhberg (Lindenstrass\cite{lin}, Rosenthal \cite{Ros}, S$\acute{a}$nchez \cite{San} and Castillo-Moreno \cite{Cas}) and from \cite{ASC} by Avil$\acute{e}$s-S$\acute{a}$nchez-Castillo-Gonz$\acute{a}$lez-Moreno to conclude a quantitative characterization of separably injective Banach spaces which completely solves Problem \ref{P}.

\begin{thm}\label{T} \rm
 (i) If $X$ is a $\lambda$-separably injective Banach space, then the pair $(X, Y)$ is $(3\lambda,6\lambda)$ stable for every separable Banach space $Y$.

 (ii) If the pair $(X, Y)$ is $(\lambda,\gamma)$ stable for every separable Banach space $Y$, then $X$ is a $\lambda$-separably injective Banach space.
\end{thm}

\begin{proof} (i) Since $Y$ is separable, it follows from Lemma \ref{main2} that for every $w^*$-dense subset $\Omega\subset$ ext $(B_{X^*})$, there is a bounded linear operator $T: Y\rightarrow \ell_{\infty}(\Omega)$ such that

$$\| T f(x)-x\|\leq 2 \varepsilon,\;\;\text{for\;all}\;x\in X .$$
Let $Z=\rm {\overline{span}}\;\{Tf(X)\cup X\}$. It follows from the continuity of $T$ that $Z/X$ is separable quotient space since $Y$ is separable. Since $X$ is $\lambda$- separable injective, it follows from Proposition \ref{P:3.2} that $X$ is $3\lambda$-complemented in $Z$. Therefore, there is a bounded linear operator $P:Z\rightarrow X$ with $\|P\|\leq 3\lambda$ such that

$$\| PT f(x)-x|=\|PT f(x)-Px\| \leq 6\lambda \varepsilon,\;\;\text{for\;all}\;x\in X ,$$
where $PT: L(f)\rightarrow X$ satisfies that $\|PT\|\leq 3\lambda$.

(ii) By Proposition \ref{P:3.2}, it suffices to show that $X$ is $\lambda$-complemented in every superspace $Y$ such that $Y/X$ is separable.
Let $Y=X+Y/X$ be the algebraic direct sum. Since $Y/X$ is separable, card $(X)$ $=$ card $(Y)$. It follows from Qian's counterexample (i.e., Lemma \ref{L:1.1}) that there is an $\varepsilon$-isometry $f: X \rightarrow Y$ such that $Y=L(f)$ and $f(0)=0$.
Hence by the assumption, there is a projection $P:Y\rightarrow X$ with $\|P\|\leq \lambda$
and we complete the proof.
\end{proof}

Recall that a compact Hausdorff space $K$ is said to be an $F$-space if disjoint open $F_\sigma$ sets have disjoint closures. For example, $\beta\mathbb{N}$, the C$\check{e}$ch-Stone compactification of $\mathbb{N}$ and $\beta\mathbb{N}\backslash\mathbb{N}$ are $F$-spaces. Since $C(K)$ is 1-separably injective for every $F$-space $K$ (see, for instance, \cite[p.202-203]{ASC}, \cite{lin}), we have

\begin{cor}\label{C1}
For every compact $F$-space $K$ (for example, $K=\beta\mathbb{N}\backslash\mathbb{N}$), the pair $(C(K), Y)$ $($resp. $(\ell_\infty/c_0, Y)$ $)$ is $(3,6)$ stable for every separable Banach space $Y$.
\end{cor}

\begin{proof} It is sufficient to note that $\ell_\infty/c_0$ is linearly isometric to
$C(\beta\mathbb{N}\backslash\mathbb{N})$.
\end{proof}

Recall that a compact space $K$ has height n if $K^{(n)}=\emptyset$, where we write $K'$ for the derived set of $K$ and $K^{(n+1)}=(K^{(n)})'$. Since $C(K)$ is $(2n-1)$-separably injective for every $K$ of height $n$ (see, for instance, \cite[p.203]{ASC}), we have

\begin{cor}\label{C1.1}
For every compact space $K$ of height $n$, the pair $(C(K), Y)$ is $(6n-3,12n-6)$ stable for every separable Banach space $Y$.
\end{cor}

To combine Theorem \ref{T} with the results of Johnson-Oikhberg \cite{John} that for every family of $\lambda$-separably injective spaces $\{E_i\}_{i\in\Lambda}$, $(\sum_{i\in\Lambda} E_i)_{\ell_\infty}$ and $(\sum_{i\in\Lambda} E_i)c_0)$ are respectively $\lambda$-separably injective and $2\lambda^2$-separably injective, which was also proved by Rosenthal \cite{Ros}, S$\acute{a}$nchez \cite{San} and Castillo-Moreno \cite{Cas} with the estimates for the constant, respectively $\lambda(1+\lambda)^{+}$, $(3\lambda^2)^+$ and $6(1+\lambda)$, we have the following corollaries.

\begin{cor}\label{C2}
The pair $((\sum_{i\in\Lambda} E_i)_{\ell_\infty}, Y)$ is $(3\lambda,6\lambda)$ stable for every separable Banach space $Y$, where $\{E_i\}_{i\in\Lambda}$ is a family of $\lambda$-separably injective spaces.
\end{cor}

\begin{cor}\label{C3}
 The pair $((\sum_{i\in\Lambda} E_i)c_0), Y)$ is $(6\lambda^2, 12\lambda^2)$ (resp. $(3\lambda (1+\lambda)^{+}, 6\lambda (1+\lambda)^{+})$, $((9\lambda^2)^{+}, (18\lambda^2)^+)$ and $(18(1+\lambda),36(1+\lambda))$ stable for every separable Banach space $Y$, where $\{E_i\}_{i\in\Lambda}$ is a family of $\lambda$-separably injective spaces.
 \end{cor}

\begin{rem} There are many other examples for separably injective Banach spaces, such as the Johnson-Lindenstrauss spaces \cite{John2}, Benyamini-space which is an M-space nonisomorphic to a $C(K)$-space \cite{ben1} and the WCG nontrivial twisted sums of $c_0(\Gamma)$ constructed by Argyros, Castillo, Granero, Jimenez and Moreno \cite{Arg} (see, for instance, \cite{ASC}).

\end{rem}

Qian \cite{qian} proved that the pair $(L_p,L_p)$ is stable for $1< p<\infty$. \v{S}emrl and V\"{a}is\"{a}l\"{a} \cite{sm} gave a sharp estimate for the constant pair $(\alpha,\gamma)$ with $\gamma=2$. Therefore, it is very natural to ask:

\begin{prob}\label{P3}
Is it true that the following pairs are stable for $1\leq p\leq\infty$ and $p\neq q<\infty$?

\rm (1) $((\sum_{n=1}^{\infty} l_p^n)_{c_0}, (\sum_{n=1}^{\infty} l_p^n)_{c_0})$;
(2)$((\sum_{n=1}^{\infty} l_p^n)_{\ell_\infty}, (\sum_{n=1}^{\infty} l_p^n)_{\ell_\infty})$;

(3) $((\sum_{n=1}^{\infty} \ell_\infty)_{l_p}, (\sum_{n=1}^{\infty} \ell_\infty)_{l_p})$;
(4) $((\sum_{n=1}^{\infty} l_p)_{\ell_\infty}, (\sum_{n=1}^{\infty} l_p)_{\ell_\infty})$;

(5) $((\sum_{n=1}^{\infty} L_p)_{\ell_\infty}, (\sum_{n=1}^{\infty} L_p)_{\ell_\infty})$;
(6) $((\sum_{n=1}^{\infty} c_0)_{l_p}, (\sum_{n=1}^{\infty} c_0)_{l_p})$;

(7) $((\sum_{n=1}^{\infty} L_p)_{c_0}, (\sum_{n=1}^{\infty} L_p)_{c_0})$;
(8) $((\sum_{n=1}^{\infty} \ell_p)_{c_0}, (\sum_{n=1}^{\infty} l_p)_{c_0})$.

(9) $((\sum_{n=1}^{\infty} l_p)_{\ell_q}, (\sum_{n=1}^{\infty} l_p)_{\ell_q})$;
(10) $((\sum_{n=1}^{\infty} L_p)_{\ell_q}, (\sum_{n=1}^{\infty} L_p)_{\ell_q})$.
\end{prob}

It is true for (1), (2), (3), (4) and (5) if $p=\infty$ as we have proved. In this case, it is not true for (6), (7) and (8) since $(\sum_{n=1}^{\infty} c_0)_{\ell_\infty}$, $(\sum_{n=1}^{\infty} L_\infty)_{c_0}$ and $(\sum_{n=1}^{\infty} \ell_\infty)_{c_0}$ are not complemented in $\ell_\infty$. If $1\leq p<\infty$, then it is also not true for (3), (4) and (5) since $(\sum_{n=1}^{\infty} \ell_\infty)_{l_p}$, $(\sum_{n=1}^{\infty} l_p)_{\ell_\infty}$ and $(\sum_{n=1}^{\infty} L_p)_{\ell_\infty}$ are not complemented in $\ell_\infty$. However, we do not know if it is true or not for the above problem \ref{P3} in general case.

\section{A quantitative and generalized Sobczyk theorem }

 If $E_i$ is a $\lambda$-injective Banach spaces for each $i\in\Lambda$ (A Banach space $X$ is said to be $\lambda$-injective if it is $\lambda$-complemented in every superspace i.e., a Banach space which contains $X$), then by Theorem \ref{T:1.3} we have the following Theorem \ref{T:3.7} which gives sharpen constants ($\alpha,\gamma$) for Qian's Problem. In some sense, it could be seen as a quantitative and generalized Sobczyk theorem (See Theorem \ref{T:1.4}).

\begin{thm}\label{T:3.7} Let $\Lambda$ and $\Gamma_i$ for each $i \in \Lambda$ are index sets. Suppose that one of the following three statements holds

 i) $X$ is isomorphic to $Z=$ $ (\sum_{i\in\Lambda} c_0(\Gamma_i))_{\ell_\infty}$ and $\lambda>d(X,Z)$;

ii) $X$ is isomorphic to $Z=$ $(\sum_{i\in\Lambda} \ell_\infty(\Gamma_i))_{c_0}$ and $\lambda>d(X,Z)$;

iii) $X=(\sum_{i\in\Lambda} E_i)c_0$ and $\{E_i\}_{i\in\Lambda}$ is a family of $\lambda$-injective Banach spaces.

Then $(X,Y)$ is $(2\lambda,4\lambda)$-stable for every separable Banach space $Y$.
\end{thm}

\begin{proof}
 i) Let $X$ be a Banach space isomorphic to $ (\sum_{i\in\Lambda} c_0(\Gamma_i))_{\ell_\infty}$ and $T: X\rightarrow (\sum_{i\in\Lambda} c_0(\Gamma_i))_{\ell_\infty}$ be an isomorphism such that $\|T\|\cdot\|T^{-1}\|<\lambda$. For each $n\in\Lambda$ and $m\in\Gamma_n$,
 let $e_{nm}\in(\sum_{i\in\Lambda} c_0(\Gamma_i))_{\ell_\infty}$ with the standard biorthogonal functionals $e_{nm}^*\in(\sum_{i\in\Lambda} c_0(\Gamma_i))_{\ell_\infty}^*$ such that $e_{ij}^*(e_{nm})=\delta_{in}\delta_{jm}$.
 For all $n\in\Lambda$ and $m\in\Gamma_n$, let $x_{nm}\in X$ be such that $T(x_{nm})=e_{nm}$. Let $T^*:Z^*\rightarrow X^*$ be the conjugate operator of $T$. Then
$$T(x)=\{\sum_{m\in\Gamma_n} (T^*e_{nm}^*)(x)e_{nm}\}_{n\in\Lambda}$$
 and
 $$x=T^{-1}\{\sum_{m\in\Gamma_n} (T^*e_{nm}^*)(x)e_{nm}\}_{n\in\Lambda},\;\;{\rm for\;all}\;x\in X.$$
  For all $n\in\Lambda$ and $m\in\Gamma_n$, let $x_{nm}^*= T^*e_{nm}^*\in \|T\|B_{X^*}$. It follows from Theorem \ref{T:1.3} that for every $n\in\Lambda$ and $m\in\Gamma_n$, there exists a functional $\phi_{nm}\in \|T\|B_{Y^*}$ with $\|\phi_{nm}\|=\|x_{nm}^*\|$ such that
\begin{align}\label{E:4.3}
|\langle\phi_{nm},f(x)\rangle-\langle x_{nm}^*,x\rangle|
\leq2 \eps\|T\|,\; {\rm for \;all}\; x\in X.\end{align}
It follows from the $w^*-w^*$ continuity of $T^*$ that for each $n\in\Lambda$, $x^*_{nm}\rightarrow0$ in the $w^*$-topology of $X^*$ Since $e^*_{nm}\rightarrow0$ in the $w^*$-topology of $Z^*$. Let
\begin{align*} K=\{\psi\in \|T\|B(Y^*):|\langle\psi, f(x)\rangle|\leq 2\eps\|T\|, \; {\rm for \;all}\; x\in X\}.\end{align*}
Then $K$ is a nonempty $w^*$-compact subset of $Y^*$. Since $Y$ is separable, $(\|T\|B_{Y^*},w^*)$ is metrizable. Let $d$ be a metric such that $(\|T\|B_{Y^*}, d)$ is homeomorphic to $(\|T\|B_{Y^*},w^*)$. Since for each $n\in\Lambda$, $(x^*_{nm})$ is a $w^*$-null net in $X^*$, inequality \eqref{E:4.3} implies that for each $n\in\Lambda$, every $w^*$-cluster point $\phi$ of $(\phi_{nm})$ is in $K$ such that $\|\phi\|\leq\|T\|$, which yields that ${\rm d}(\phi_{nm},K)\rightarrow0$ for each $n\in\Lambda$. Hence, for each $n\in\Lambda$, there is a net $(\psi_{nm})\subset K$ such that  d$(\phi_{nm},\psi_{nm})\rightarrow0$, or equivalently, $\phi_{nm}-\psi_{nm}\rightarrow0$ in the $w^*$-topology of $Y^*$. Let $S:Y\rightarrow X$ be defined for every $y\in Y$ by
\begin{align*}S(y)= T^{-1}\{\sum_{m\in\Gamma_n}\langle \phi_{nm}-\psi_{nm},y\rangle e_{nm}\}_{n\in\Lambda}\in X.\end{align*} Hence
$$\|S\|\leq2\|T\|\cdot\|T^{-1}\|<2\lambda$$ and

\begin{align}&\|Sf(x)-x\|=\|T^{-1}\{\sum_{m\in\Gamma_n} \langle\phi_{nm}-\psi_{nm},f(x)\rangle e_{nm}\}_{n\in\Lambda}-T^{-1}\{\sum_{m\in\Gamma_n}
 \langle x_{nm}^*,x\rangle e_{nm}\}_{n\in\Lambda}\|\notag\\
 &\leq\|T^{-1}\|\sup_{n\in\Lambda}(\|\sum_{m\in\Gamma_n} \langle\phi_{nm}-\psi_{nm},f(x)\rangle e_{nm}-\sum_{m\in\Gamma_n}
 \langle x_{nm}^*,x\rangle e_{nm}\|)\notag\\
&\leq\|T^{-1}\|\cdot \sup_{n\in\Lambda}\sup_{m\in\Gamma_n}| \langle\phi_{nm},f(x)\rangle-\langle x_{nm}^*,x\rangle -\langle\psi_{nm},f(x)\rangle |\notag\\
&\leq\|T^{-1}\|(\sup_{n\in\Lambda}\sup_{m\in\Gamma_n} |\langle\phi_{nm},f(x)\rangle-\langle x_{nm}^*,x\rangle| +\sup_{n\in\Lambda}\sup_{m\in\Gamma_n}|\langle\psi_{nm},f(x)\rangle|)\notag\\
&\leq 4 \eps\|T\|\cdot\|T^{-1}\|<4\eps\lambda.\notag\end{align}

ii-iii) For each $i\in\Lambda$, $\Gamma_i$ denotes by $B_{E_i^*}$. It suffices to show this case that $X=(\sum_{i\in\Lambda} E_i)c_0$. Let $J: X=(\sum_{i\in\Lambda} E_i)c_0\rightarrow (\sum_{i\in\Lambda}\ell_\infty (B_{E_i^*}))c_0=(\sum_{i\in\Lambda}\ell_\infty (\Gamma_i))c_0$ be the canonical embedding. For each $n\in\Lambda$, let $Q_n:(\sum_{i\in\Lambda}\ell_\infty (\Gamma_i))c_0\rightarrow\ell_\infty (\Gamma_n)$ be the canonical projection. Let $P_n:\ell_\infty (\Gamma_n) \rightarrow E_n$ be a family of projections with $\|P_n\|\leq \lambda$.
 For each $n\in\Lambda$ and $m\in\Gamma_n$, let $e_{nm}\in(\sum_{i\in\Lambda}\ell_\infty(\Gamma_i))_{c_0}$ with the standard biorthogonal functionals $e_{nm}^*\in((\sum_{i\in\Lambda}\ell_\infty(\Gamma_i))_{c_0})^*$ such that $e_{ij}^*(e_{nm})=\delta_{in}\delta_{jm}$.
 Then
$$x=\sum_{n\in\Lambda} \{(e_{nm}^*)(x)\}_{m\in\Gamma_n} \;\;{\rm for\;all}\;x\in X.$$
By Theorem \ref{T:1.3}, for each $n\in\Lambda$ and $m\in\Gamma_n$, there exists
$\phi_{nm}\in B_{Y^*}$ with $\|\phi_{nm}\|=\|e_{nm}^*\|$ such that

\begin{align*}
|\langle\phi_{nm},f(x)\rangle-\langle e_{nm}^*,x\rangle|
\leq2 \eps,\; {\rm for \;all}\; x\in X.\end{align*}
Clearly, $e^*_{nm}\rightarrow0$ uniformly for each $m\in\Gamma_n$ in the $w^*$-topology of $Z^*$. Let
\begin{align*}K=\{\psi\in B(Y^*):|\langle\psi, f(x)\rangle|\leq 2\eps, \; {\rm for \;all}\; x\in X\}.\end{align*}
 Since $\Gamma_n$ can be well ordered for every $n\in \Lambda$, we write $$\Gamma_n=\{0,1,2,\cdots, w_0, w_0+1, \cdots, w_1,\cdots \prec\Gamma_n  \},$$ where  $\Gamma_n$ also denotes by its ordinal number.
 It follows from i) that for each $n\in \Lambda$, there is a net $(\psi_{n0})\subset K$ such that d$(\phi_{n0},\psi_{n0})\rightarrow0$. We can choose $(\psi_{nm})\subset K$ such that for every $n\in \Lambda$ and $m\in\Gamma_n$,
 d$(\phi_{nm}, \psi_{nm})\leq$ d$(\phi_{n0},\psi_{n0})$ or equivalently, $(\phi_{nm}-\psi_{nm})\rightarrow0$ uniformly for each $m\in\Gamma_n$ in the $w^*$-topology of $Y^*$. Let $Q:Y \rightarrow (\sum_{i\in\Lambda}\ell_\infty (\Gamma_i))c_0$ be defined for all $y\in Y$ by
\begin{align*}Q(y)=\sum_{n\in \Lambda}\{\langle\phi_{nm}-\psi_{nm},y\rangle \}_{m\in\Gamma_n} \in (\sum_{i\in\Lambda}\ell_\infty (\Gamma_i))c_0, \end{align*} which yields that \begin{align*}\|Q(y)\|\leq(\sup_{n\in\Lambda, m\in\Gamma_n}\|\phi_{nm}-\psi_{nm}\|)\|y\|\leq2\|y\|.\end{align*}
 Thus $$\|Q\|\leq2.$$
Let $S:Y\rightarrow X$ be defined for all $y\in Y$ by
\begin{align*}S(y)=\sum_{n\in\Lambda} P_nQ_nQ(y)
=\sum_{n\in\Lambda}P_n\{\langle \phi_{nm}-\psi_{nm},y\rangle\}_{m\in\Gamma_n}. \end{align*}
Hence
$$\|S\|=\sup_{n\in\Lambda} \|P_nQ_nQ\|\leq2\lambda$$ and

\begin{align}&\|Sf(x)-x\|=\|\sum_{n\in\Lambda} P_n\{\langle\phi_{nm}-\psi_{nm},f(x)\rangle \}_{m\in\Gamma_n}-\sum_{n\in\Lambda}
 P_n\{\langle e_{nm}^*,x\rangle \}_{m\in\Gamma_n}\|\notag\\
&\leq\lambda\sup_{n\in\Lambda}\sup_{m\in\Gamma_n}|\langle\phi_{nm},f(x)\rangle-\langle e_{nm}^*,x\rangle-\langle\psi_{nm},f(x)\rangle|\notag\\
&\leq\lambda(\sup_{n\in\Lambda}\sup_{m\in\Gamma_n} |\langle\phi_{nm},f(x)\rangle-\langle e_{nm}^*,x\rangle| +\sup_{n\in\Lambda}\sup_{m\in\Gamma_n}|\langle\psi_{nm},f(x)\rangle|)\notag\\
&\leq 4\eps\lambda.\notag\end{align}
Thus, our proof is completed.

\end{proof}

\section{A generalized Figiel theorem for $\mathcal{L}_{\infty,\lambda}$-spaces}

Recall that a Banach space $X$ is said to be a $\mathcal{L}_{\infty,\lambda}$-space if every finite dimensional subspace $F$ of $X$ is contained in another finite dimensional subspace $E$ of $X$ such that $d(E, \ell_\infty^{\dim E})\leq\lambda$. In this case, a Banach space $X$ is said to be a $\mathcal{L}_{\infty}$-space if for some $\lambda>0$, $X$ is a $\mathcal{L}_{\infty,\lambda}$-space(see, for instance, \cite{AH}, \cite {ASC}, \cite{Bou}). For example, a $\lambda$-separably injective Banach space is a $\mathcal{L}_{\infty,9\lambda^+}$-space (see \cite [ p.199, Prop.3.5 (a)]{ASC})and every $C(K)$-space is also a $\mathcal{L}_{\infty}$-space.

\begin{thm}\label{T:1.2}
Suppose that $X$ is a $\mathcal{L}_{\infty,\lambda}$-space and $Y$ is a Banach space.  Then
for every standard $\eps$-isometry $f:X \rightarrow Y$, there is a bounded linear operator $T: Y\rightarrow X^{**}$ such that $Tf-Id$ is uniformly bounded by $ 2 \lambda\varepsilon$ on $X$.

\end{thm}

\begin{proof} By Lemma \ref{main2}, for every $w^*$-dense subset $\Omega\subset$ ext $(B_{X^*})$ there is a bounded linear operator $S: Y\rightarrow \ell_\infty(\Omega)$ with norm one such that

$$\|S f(x)-x\|\leq 2\varepsilon,\;\;\text{for\;all}\;x\in X .$$

Let $X=\cup_{i\in I} E_i$ such that for every $i, j\in (I,\succeq)$, $i\succeq j$ if and only if $E_i\supseteq E_j$ satisfying that for each $i\in I$, $\dim E_i<\infty$ and $d(E_i, \ell_\infty^{\dim E_i})\leq\lambda$. Hence for each $i\in I$, there exists a projection $P_i:\ell_\infty(\Omega)\rightarrow E_i$ such that $\|P_i\|<\lambda+\frac{1}{1+\dim E_i}$. Since $\{P_i\}_{i\in I}$ is uniformly bounded on $B_{\ell_\infty(\Omega)}$, it follows from the Arzel$\grave{a}$-Ascoli theorem that there is a subnet $\{\delta_i\}_{i\in\Lambda}$ of $I$ for an partial order set $\Lambda$ such that $P: \ell_\infty(\Omega)\rightarrow X^{**}$ is well defined by
$$P(y)=w^*-\lim_{i\in\Lambda}P_{\delta_i}(y),\;\text {for\;all}\; y\in \ell_\infty(\Omega),$$
which yields that $$\|P\|\leq \lambda \;\text {and}\; P|_X=Id.$$
Hence
 $$\| Tf(x)-x\|\leq 2\varepsilon\lambda,\;\;\text{for\;all}\;x\in X ,$$
where $T=PS: Y\rightarrow X^{**}$ with $\|T \|\leq\lambda$.
\end{proof}

\section{Acknowledgements}
This work was partially done based on many discussions with Professor W.B. Johnson while the author was visiting Texas A$\&$M University and in Analysis and Probability Workshop at Texas A$\&$M University which was funded by NSF Grant. The author would like to thank Professor W.B. Johnson and Professor Th. Schlumprecht for the invitation. This work is also a part of the author's Ph.D. thesis under the supervision of Professor Lixin Cheng.

The author was Supported by the Natural Science Foundation of China (Grant No. 11601264) and the Outstanding Youth Scientific Research Personnel Training Program of Fujian Province and the High level Talents Innovation and Entrepreneurship Project of Quanzhou City and the Research Foundation of Quanzhou Normal University(Grant No. 2016YYKJ12).

\end{document}